\documentclass[final]{mcom-l}
\newtheorem{theorem}{Theorem}[section]
\newtheorem{lemma}[theorem]{Lemma}
\newtheorem{corollary}[theorem]{Corollary}
\numberwithin{equation}{section}
\usepackage{amssymb, url}
\usepackage{hyperref}

\DeclareMathOperator{\grad}{grad}
\DeclareMathOperator{\curl}{curl}
\DeclareMathOperator{\im}{im} 
\newcommand{\sn}[1]{\mathbb{#1}}
\newcommand{\norm}[1]{\|#1\|} 
\newcommand{\ip}[1]{\langle #1 \rangle}

\newcommand{\ph}{\,\cdot\,}
\newcommand{\B}{\mathfrak{B}}
\newcommand{\h}{\mathfrak{H}}
\newcommand{\z}{\mathfrak{Z}}
\newcommand{\Div}{\operatorname{div}}

\begin{document}
\title[FEEC with lower-order terms]{Finite element exterior calculus with
  lower-order terms}
\author{Douglas N. Arnold} 
\address{School of Mathematics, University of Minnesota, Minneapolis,
  Minnesota 55455}
\email{arnold@umn.edu}
\author{Lizao Li}
\address{School of Mathematics, University of Minnesota, Minneapolis,
  Minnesota 55455}
\email{lixx1445@umn.edu}
\thanks{The work of both authors were supported in part by NSF grant
  DMS-1418805.}
\date{September 21, 2015}

\begin{abstract}
  The scalar and vector Laplacians are basic operators in physics and
  engineering. In applications, they frequently show up perturbed by
  lower-order terms. The effect of such perturbations on mixed finite
  element methods in the scalar case is well-understood, but that in the
  vector case is not. In this paper, we first show that, surprisingly, for
  certain elements there is degradation of the convergence rates with
  certain lower-order terms even when both the solution and the data are
  smooth. We then give a systematic analysis of lower-order terms in mixed
  methods by extending the Finite Element Exterior Calculus (FEEC)
  framework, which contains the scalar, vector Laplacian, and many other
  elliptic operators as special cases. We prove that stable mixed
  discretization remains stable with lower-order terms for sufficiently
  fine discretization. Moreover, we derive sharp improved error estimates
  for each individual variable. In particular, this yields new results for
  the vector Laplacian problem which are useful in applications such as
  electromagnetism and acoustics modeling. Further our results imply many
  previous results for the scalar problem and thus unifies them all under
  the FEEC framework.
\end{abstract}

\maketitle

\section{Introduction}
The vector Laplace equation, and, more generally, the Hodge Laplace
equation associated to a complex, arise in many applications. The
discretization of such equations is a basic motivation of the Finite
Element Exterior Calculus (FEEC)~\cite{Arnold2006, Arnold2010}. In many
applications the equations include variable coefficients and lower-order
terms. While the former is included in the standard FEEC framework through
weighted inner products, the latter is not, which is the subject of this
work. One might expect that lower-order perturbations degrade neither the
stability nor the convergence rates of stable Galerkin methods. However,
this need not to be true. While stable choices of finite elements for the
unperturbed Hodge Laplacian remain stable for the perturbed equation, we
find that certain lower order perturbations result in decreased rates of
convergence. Other choices of element pairs or perturbations do not lower
the convergence rate. The situation is subtle.

First, to fix ideas, we consider a simple example taken from
magnetohydrodynamics \cite[Chapter 3]{Gerbeau2006}: given vector fields
$f,v$ on a domain $\Omega\subset \sn{R}^3$, find a vector field $B$
satisfying:
\begin{gather*}
  \curl \curl B -\curl (v\times B)= f, \qquad \Div B= 0, 
  \quad \text{in $\Omega$},\\
  B\cdot n = 0, \quad (\curl B-v\times B)\times n = 0,
  \quad \text{on $\partial\Omega$}.
\end{gather*}
Physically, $B$ is the non-dimensionalized magnetic field inside a
conductor moving with a velocity field $v$. The system admits a solution
only when $\Div f=0$. In that case, the solution also satisfies the
following vector Laplace equation, which is solvable for any data and hence
more suitable for discretization:
\begin{equation}
  \label{eq:mhd}
  -\grad \Div B + \curl\curl B- \curl (v\times B)  = f.
\end{equation}
For a mixed method, we introduce $\sigma=\curl B-v\times B$ and solve the
coupled system:
\begin{equation*}
  \sigma - \curl B + v\times B = 0, \qquad
  \curl \sigma - \grad\Div B = f.
\end{equation*}
A common stable choice of mixed elements, at least when $v=0$, seeks
$B$ in the space of N\'ed\'el\'ec face elements of the second kind of degree
$r\geq 1$, and $\sigma$ in the space of N\'ed\'el\'ec edge elements of the
second kind of degree $(r+1)$ \cite{Nedelec1980, Nedelec1986}. In this
case, for the unperturbed problem, that is when $v=0$, the convergence for
the $L^2$-error in $\sigma$ is of optimal order $O(h^{r+2})$ if the
solution is smooth enough. However, as we show in Section~\ref{sec:mainres}
and verify by numerical computation in Section~\ref{sec:numexp}, when $v$
does not vanish, the $L^2$-convergence for $\sigma$ is reduced to order
$O(h^{r+1})$. A similar phenomenon was observed in mixed methods for the
scalar Laplacian by Demlow~\cite{Demlow2002}. But the vector case we study
has more surprises. For example, consider the vector Laplacian perturbed by
a zeroth-order term: for some real coefficient $A$,
\begin{gather*}
  -\grad \Div u + \curl\curl u + Au  = f, 
  \quad \text{in $\Omega$},\\
  u\times n = 0, \quad \Div u = 0,
  \quad \text{on $\partial\Omega$}.
\end{gather*}
This problem arises, for example in electromagnetism where $u$ is the
electric field and $A$ is the conductivity coefficient. If we use the same
mixed finite element method just considered, here with $\sigma=\curl u$,
then the $L^2$ error in $\sigma$ is one order suboptimal for a general
matrix coefficient $A$, but optimal if $A$ is a scalar coefficient.

We now summarize convergence rates derived from our main abstract theorems
applied to the perturbed Hodge Laplace problem, of which the previous
vector Laplace problem is an instance. These are important in and directly
relevant to many applications. In particular, our results for the vector
cases, that is, $1$-forms in $2$D and $1$- and $2$-forms in $3$D, are new. On a
domain $\Omega$ in $\mathbb R^n$, a $k$-form has $\binom{n}{k}$
coefficients and the meaning of the exterior derivative $d$ and the
codifferential $\delta$ depends on $k$. For example, when $n=3$ and $k=2$,
we have the case considered before with
$\delta d + d\delta = -\grad \Div + \curl \curl$. For any $0\leq k\leq n$,
the $k$-form Hodge Laplace equation seeks a $k$-form $u$ satisfying:
\begin{equation}
  \label{eq:unperturbed-hodge-laplacian}
  L_0 u := (d\delta+\delta d)u=f,
\end{equation}
with proper boundary conditions. Our abstract theory applies to equation
\eqref{eq:unperturbed-hodge-laplacian} perturbed by general lower order
terms: 
\begin{equation*}
  Lu = [(d+l_3)(\delta+{l_2})+(\delta +{l_4})(d+l_1) +l_5]u = f,
\end{equation*}
with proper boundary conditions. In Section~\ref{sec:mainres}, we allow
$l_i$ to be general linear operators, but here we assume that they are
multiplication by smooth coefficient fields. For example, $l_1$ takes a
$k$-form to $(k+1)$-form, thus may be viewed as a
$\binom{n}{k+1}\times \binom{n}{k}$ matrix field.

The mixed formulation solves simultaneously for $u$ and for
$\sigma=(\delta + l_2)u$, which is a $(k-1)$-form.  On a simplicial
triangulation of $\Omega$, for any $r\ge 1$, we have four canonical pairs
of mixed finite elements for $(\sigma, u)$:
\begin{equation*}
  \mathcal{P}_{r+1}\Lambda^{k-1}\times \mathcal{P}_{r}\Lambda^k, \quad
  \mathcal{P}^-_{r+1}\Lambda^{k-1}\times \mathcal{P}_{r}\Lambda^k,  \quad
  \mathcal{P}_{r}\Lambda^{k-1}\times \mathcal{P}_{r}^-\Lambda^k,  \quad
  \mathcal{P}^-_{r}\Lambda^{k-1}\times \mathcal{P}^-_{r}\Lambda^k.
\end{equation*}
In dimension $\leq 3$, all these are classical mixed finite elements.  For
example, the pair consisting of N\'ed\'el\'ec face elements of the second kind and
N\'ed\'el\'ec edge elements of the second kind before is
$\mathcal{P}_{r+1}\Lambda^1\times \mathcal{P}_{r}\Lambda^2$. For more, the
Periodic Table of the Finite Elements (\url{http://femtable.org/}) collects
all these elements and their correspondence to classical elements. A main
result in FEEC is that the above four pairs lead to stable mixed finite
element methods for the unperturbed Hodge Laplace problem. Further, the
rates of convergence for the $L^2$-errors in $\sigma$, $d\sigma$, $u$, and
$du$ are optimal, as determined by the approximation properties of the
spaces.  Thus, for example, if the pair
$\mathcal{P}_{r+1}\Lambda^{k-1}\times \mathcal{P}_{r}\Lambda^k$ is used,
the order of convergence for the $L^2$-error in $\sigma$, $d\sigma$, $u$,
and $du$ are $r+2$, $r+1$, $r+1$, and $r$, respectively, while, for
$\mathcal{P}^-_{r}\Lambda^{k-1}\times \mathcal{P}^-_{r}\Lambda^k$ all four
$L^2$-errors converge with order $r$.

For the perturbed system, our discrete stability result
(Theorem~\ref{thm:stability}) implies that when the perturbed problem is
uniquely solvable at the continuous level (which is the generic case, as we
will prove), then the mixed discretization of this problem using any of the
four stable pairs before is still stable for sufficiently fine
discretization.  Further, if we have full elliptic regularity (for example
on a smooth domain \cite{Gaffney1951}), then our improved error estimate
(Theorem~\ref{thm:improved_est}) implies that the $L^2$ convergence rates
for the unperturbed Hodge Laplacian still hold for the perturbed problem
with a few exceptions, as summarized in Table~\ref{tb:conv}:
\begin{table}[h!]
  \caption{$L^2$-error rates for FEEC elements solving Hodge Laplace equation}
  \label{tb:conv}
  \begin{tabular}{ c | c | c | c | c } 
    \hline
    elements& $\sigma$ & $d\sigma$ & $u$ & $du$ \\ \hline
    $\mathcal{P}_{r+1}\Lambda^{k-1}\times \mathcal{P}_{r}\Lambda^k$ 
    & 
      $\scriptstyle
      \begin{cases}
        {r+2}, & \text{if no $l_2$, $l_4$, $l_5$},\\ 
        r+1,   & \text{otherwise.}
      \end{cases}$ 
    & $\scriptstyle
      \begin{cases}
        {r+1}, & \text{if no $l_4$},\\ 
        r, & \text{otherwise.}
      \end{cases}$ 
    & ${r+1}$ & $r$ \\
    $\mathcal{P}^-_{r+1}\Lambda^{k-1}\times \mathcal{P}_{r}\Lambda^k$ 
    & ${r+1}$
    & $\scriptstyle
      \begin{cases}
       {r+1}, & \text{if no $l_4$},\\ 
       r, & \text{otherwise.}
       \end{cases}$ 
    & ${r+1}$ & $r$ \\
    $\mathcal{P}_{r}\Lambda^{k-1}\times \mathcal{P}_{r}^-\Lambda^k$ 
    & $\scriptstyle
      \begin{cases}
       {r+1}, & \text{if no $l_2$, $l_5$},\\
       r, & \text{otherwise.}
       \end{cases}$ 
    & ${r}$& ${r}$& $r$ \\
    $\mathcal{P}^-_{r}\Lambda^{k-1}\times \mathcal{P}^-_{r}\Lambda^k$ 
    &$r$ & $r$ &$r$ & $r$\\ \hline
  \end{tabular}
\end{table}

For example, in our first example, $B$ is a 2-form and equation
\eqref{eq:mhd} has an $l_2$ lower-order term, which leads to reduced
$L^2$-error rates for $\sigma$. In practice, this suggests that the smaller
space $\mathcal{P}_{r+1}^-\Lambda^k$ should be used for $\sigma=\curl B$,
because the use of the bigger space $\mathcal{P}_{r+1}\Lambda^k$ does not
improve the convergence rates for $\sigma$ or any other quantity. In our
second example, the zeroth-order term is a generic $l_5$ term, hence it
also degrades the $L^2$-error rate in $\sigma$.  We observe that some
lower-order terms, namely $l_1$ and $l_3$ terms, do not degrade the error
rates in any of the cases above and that the $L^2$-convergence rates for
$u$ and $du$ are unaffected by the lower-order terms. We also note that,
most surprisingly, the lower-order term $l_4$ has the worst effect on the
convergence rates of the error in $\sigma$, degrading the $L^2$-error rates
in both $\sigma$ and $d\sigma$, yet $\sigma$ has no apparent dependence on
$l_4$. 

Historically, the effect of lower-order terms on the convergence of finite
element methods was first studied by Schatz~\cite{Schatz1974}, for the scalar
Laplace equation with Lagrange elements. The key tool in that analysis was the
Aubin-Nitsche duality argument, which was introduced to prove $L^2$-error
estimates in both non-mixed \cite{Nitsche1968} and mixed methods
\cite{Falk1980}. These ideas guide the current techniques. Directly relevant to
this work are the studies by Douglas and Roberts \cite{Douglas1982,
  Douglas1985} and Demlow \cite{Demlow2002} on mixed finite element
discretization of the scalar Laplace equation, which is the Hodge Laplace
equation for $n$-forms in $n$ dimensions.  The primary mixed finite elements
for this problem are the Raviart--Thomas (RT) family
$\mathcal P_r^-\Lambda^{n-1}$ \cite{Raviart1977} and the BDM family
$\mathcal P_r\Lambda^{n-1}$ \cite{Brezzi1985, Brezzi1987}.  Douglas and Roberts
proved the optimal $L^2$ convergence rates for both variables when the RT
family is used.  Demlow showed that for BDM elements, even for constant
coefficients and smooth solutions, there is degradation of the convergence rate
for $\sigma$ in the problem $-\Div (\grad u+ \vec{b}u)+cu=f$ while there is no
degradation if the same problem is formulated as
$-\Div\grad u + \vec{b}\cdot \grad u+c u=f$.  All these classical results on
non-mixed and mixed finite elements for scalar Laplace problem with lower-order
terms can be read off directly from Table~\ref{tb:conv}. Our approach here
gives a uniform derivation of all classical $L^2$-estimates for the
scalar/vector Laplacian perturbed by lower-order terms under the extended
abstract FEEC framework.

The rest of this paper is organized as follows. We first briefly review the
basic FEEC framework for the unperturbed abstract Hodge Laplacian in
Section~\ref{sec:review}. Then we lay out our extended abstract framework
and state our two main discrete result: stability theorem
(Theorem~\ref{thm:stability}) and improved error estimates
(Theorem~\ref{thm:improved_est}) in Section~\ref{sec:mainres}. After that,
we prove the well-posedness theorems at the continuous level in
Section~\ref{sec:continuous_wp}. Then we prove the two main discrete
results in Section~\ref{sec:stability} and Section~\ref{sec:improved_est}
respectively. Finally, in Section~\ref{sec:numexp}, we show that the
estimates are sharp for the Hodge Laplacian case through numerical
examples.

\section{Review of the abstract FEEC fraemwork}
\label{sec:review}
FEEC is an abstract framework for analyzing mixed finite element
methods~\cite{Arnold2006,Arnold2010}.  A \emph{Hilbert complex} $(W^k,d^k)$
is a sequence of Hilbert spaces $W^k$ and closed densly-defined linear
operators $d^k:W^k\rightarrow W^{k+1}$ with closed range satisfying
$d^{k+1} \circ d^k=0$.  We use $(\,\cdot\,,\,\cdot\,)$ to denote the
$W$-inner product and $\norm{\,\cdot\,}$ to denote the $W$-norm.  Let
$d_{k+1}^*$ be the adjoint of $d^k$, $V^k=D(d^k)$, and
$V^*_{k+1}:=D(d^*_{k+1})$. From this definition,
$(d^k u,v)=(u,d^*_{k+1} v)$, for all $u\in V^k$ and $v\in V^*_{k+1}$.  One
important structure is the \emph{Hodge decomposition}:
\begin{equation*}
  W^k = \z^k \oplus (\z^k)^{\perp W}
  =\B^k \oplus \h^k \oplus (\z^k)^{\perp W}
  =\B^k \oplus \h^k \oplus \B_k^*,
\end{equation*}
where $\z^k=\ker d^k$, $\B^k=\im d^{k-1}$, $\h^k=\z^k\cap (\B^k)^{\perp}$,
and $\B_k^*=\im d^*_{k+1}=(\z^k)^{\perp W}$. 

We assume the Hilbert complex satisfies the \emph{compactness property},
that $V^k\cap V_k^*$ is a compact densely embedded subspace of $W$. This is
true for all the cases we are interested in. For example, the de~Rham
complex on Lipschitz domains in $\sn{R}^n$ satisfies this
property~\cite{Picard1984}.

The abstract Hodge Laplacian is the unbounded operator
$L_0^k:W^k\rightarrow W^k$ defined by $L_0^k = d^{k-1}d^*_k+d^*_{k+1}d^k$
with the domain
$D(L_0^k)=\{u\in V^k\cap V_k^*\mid du \in V_{k+1}^*, d^*u \in V^{k-1}\}$.
In the following, we drop the index $k$ when it is clear from the
context. For example, $L_0 = dd^*+ d^*d$. It is known that for any
$f\in W^k$, there exists a unique $u\in D(L_0)$ such that
\begin{equation*}
  L_0u = f\mod \h, \quad u\perp \h.
\end{equation*}
Let $K_0:f\mapsto u$ be the solution operator above. It is known that $K_0$
is self-adjoint and compact as a map $W\rightarrow W$.

The mixed discretization of the abstract Hodge Laplacian is
well-understood. The problem above can be formulated in the \emph{mixed
  weak formulation}: given $f\in W$, find
$(\sigma, u, p)\in V^{k-1}\times V^k\times \h^k$ such that
\begin{equation}
  \label{eq:m0}
  \begin{aligned}
    &(\sigma, \tau) - (u,d\tau) =0, &&\forall \tau \in V^{k-1}, \\
    &(d\sigma, v) + (du, dv) + (p,v)= (f,v), && \forall v \in V^k,\\
    &(u, q) = 0,&&\forall q\in \h^k.
  \end{aligned}
\end{equation}
For each $k$, let $V_h^k$ be a sequence of discrete subspaces of $V^k$
indexed by $h$. $V_h^k$ is called \emph{dense} if
\begin{equation*}
  \forall u \in V^k,\quad 
  \lim_{h\rightarrow 0} \inf_{v\in V_h}\norm{u-v}_V=0.  
\end{equation*}
Clearly, density is necessary for $V_h^k$ to be good approximations of
$V^k$. However, it is known that this alone is not sufficient for the
Galerkin projection to be a convergent method. The key additional
properties are the \emph{subcomplex property} that
$dV^k_h\subset V^{k+1}_h$ and the existence of \emph{$V$-bounded cochain
  projections}, that is, bounded projections $\pi_h^k:V^k\rightarrow V_h^k$
satisfying $d\pi_h=\pi_hd$. The main result in FEEC states that the
Galerkin projection of system~\eqref{eq:m0} using dense discrete subspace
admitting bounded cochain projections is stable. More precisely, Theorem
3.8 of~\cite{Arnold2010} states that the bilinear form associated with the
system~\eqref{eq:m0}
\begin{equation*}
  B_0((\sigma, u, p), (\tau,v, q)) := 
  (\sigma, \tau) - (u,d\tau) + (d\sigma, v) + (du, dv) +(p,v)-(u,q),
\end{equation*}
satisfies the inf-sup condition on
$(V^{k-1}_h\times V^k_h\times \h^k_h)^2$. This implies quasi-optimal
convergence rates in the $V$-norm. Further if the bounded projections can
be extended to \emph{$W$-bounded cochain projections}, that is, the
extension to $\pi_h^k:W^k\rightarrow V_h^k$ exists and
$\norm{\pi_h^k}_{W\rightarrow W}$ are bounded uniformly in $h$, then we get
decoupled error estimates estimates for each variable in the $W$-norm. To
state this precisely, we need some notations. We use $a\lesssim b$ to
express that $a\le C b$ for some generic constant $C$.
Following~\cite{Arnold2010}, we let
\begin{equation}
  \label{eq:approx0}
  \begin{gathered}
    \delta_0 = \norm{(I-\pi_h)K_0}_{W^k\rightarrow W^k}, \qquad
    \mu_0 = \norm{(I-\pi_h)P_\h}_{W^k\rightarrow W^k},\\
    \eta_0 = \max_{j=0,1}\{\norm{(I-\pi_h)dK_0}_{W^{k-j}\rightarrow W^{k-j+1}}, 
    \norm{(I-\pi_h)d^*K_0}_{W^{k+j}\rightarrow W^{k+j-1}}\},\\
    \alpha_0 = \eta_0^2+\delta_0+\mu_0,
  \end{gathered}
\end{equation}
where $P_{\h}$ is the $W$-orthogonal projection from $W$ to $\h\subset W$.

All these quantities converge to $0$ as $h\rightarrow 0$ due to the compactness
and density assumption. For example on smooth or convex polyhedral domains for
the de~Rham complex, it is known that $\eta_0=O(h)$,
$\delta_0=O(h^{\min(2,r+1)})$, $\mu_0=O(h^{r+1})$, where $r$ is the largest
degree of complete polynomials in $V_h$. We use the following notation for the
best approximation: for $w\in V^k$,
\begin{equation*}
  E(w):=\inf_{v\in V^k_h} \norm{w-v}.
\end{equation*}
Theorem 3.11 of~\cite{Arnold2010} bounds the $W$-norm error of each
variable in terms of best approximation errors:
\begin{equation}
  \label{eq:werr0}
  \begin{gathered}
    \norm{\sigma-\sigma_h}\lesssim E(\sigma)
    +\eta_0 E(d\sigma),\\
    \norm{d(\sigma-\sigma_h)}\lesssim E(d\sigma), \\
    \norm{u-u_h} \lesssim E(u)
    +(\eta_0^2+\delta_0)[E(d\sigma)+E(p)]
    +\eta_0[E(du)+E(\sigma)], \\
    \norm{d(u-u_h)}\lesssim E(du)
    +\eta_0[E(d\sigma)+E(p)],\\
    \norm{p-p_h}\lesssim E(p)+\mu_0 E(d\sigma).
  \end{gathered}
\end{equation}
For example, using
$\mathcal{P}_{r+1}\Lambda^{k-1}\times\mathcal{P}_r\Lambda^k$ and assuming
full elliptic regularity, we have $E(\sigma)= O(h^{r+2})$ and
$E(d\sigma)= O(h^{r+1})$ for any $r\geq 1$.  Therefore, we get the optimal
rates $\|\sigma-\sigma_h\|= O(h^{r+2})$. Similarly, the error rates for all
variables can be shown to be optimal for all four canonical FEEC element
pairs. Let $P_h:W\rightarrow V_h$ be the $W$-orthogonal projection and
$K_{0h}:V^k_h\to V^k_h$ be the discrete solution operator
$K_{0h}:f\mapsto u_h$. These convergence estimates can be stated using
operators (cf. Corollary 3.17 of~\cite{Arnold2010}):
\begin{equation}
  \label{eq:est-op}
  \begin{gathered}
    \norm{K_0-K_{0h}P_h}_{W\rightarrow W}\lesssim \alpha_0,\\
    \norm{dK_0-dK_{0h}P_h}_{W\rightarrow W}+
    \norm{d^*K_0-d^*_hK_{0h}P_h}_{W\rightarrow W}\lesssim \eta_0.    
  \end{gathered}
\end{equation}
The FEEC approach not only gives optimal error rates, but also captures the
important structures of the Hodge Laplacian. We have the \emph{discrete
  Hodge decomposition}:
\begin{displaymath}
  V^k_h=\z^k_h\oplus\B^*_{k,h}=\B^k_h\oplus\h^k_h\oplus\B^*_{k,h}.
\end{displaymath}
The map $\pi_h^k$ ensures that $\z_h\subset \z$ and $\B_h \subset \B$, but
the discrete spaces $\h_h$ and $\B^*_h$ are generally not subspaces of
their continuous counterparts. At the continuous level, we have
\begin{equation}
  \label{eq:hodge}
  \forall f \in W, \qquad
  f= dd^*K_0f + d^*dK_0f + P_\h f \in \B\oplus \B^* \oplus \h,
\end{equation}
that is, solving the Hodge Laplacian problem with data $f$ leads to the
Hodge decomposition of $f$. At the discrete level, similarly, we have:
\begin{equation}
  \label{eq:hodgeh}
  \forall f\in V_h, \qquad
  f= dd^*_hK_{0h}f + d^*_hdK_{0h}f +P_\h f \in  
  \B_h\oplus \B^*_h \oplus \h_h.
\end{equation}
This orthogonality is the key ingredient in deriving decoupled $W$-norm
error estimates for each variable above and plays an important role in our
analysis as well. 

\section{Main results}
\label{sec:mainres}
We start by identifying $W^k$ with its dual and form a Gelfand triple
$V^k\cap V^*_k\subset W^k \subset (V^k\cap V^*_k)'$. We extend
$L_0=dd^*+d^*d$ to an operator $V\cap V^*\rightarrow (V\cap V^*)'$ which is
more suitable for studying perturbations: for all $u,v\in V^k\cap V^*_k$,
\begin{equation*}
  \ip{L_0u,v} := (d^*u,d^*v) + (du, dv).
\end{equation*}
Our main operator $L:V^k\cap V^*_k\rightarrow (V^k\cap V^*_k)'$ is obtained
by perturbing each abstract differential: for all $u,v\in V^k\cap V^*_k$,
\begin{equation}
  \label{eq:Lw}
  \ip{Lu,v} := ((d^*+l_2) u, (d^*+l_3^*) v) + ((d+l_1)u, dv) 
  + (l_4du ,v) + (l_5u, v),
\end{equation}
where $l_i:W\rightarrow W$ are bounded linear maps between appropriate
levels for $i=1,\ldots, 5$. More succinctly, we write,
\begin{equation*}
  L=(d+l_3)(d^*+l_2) u + d^*(d+l_1)u + l_4du + l_5u.
\end{equation*}
The grouping is convenient for the mixed form later. We prove in
Lemma~\ref{lem:almost_good} of the next section that this $L$ satisfies a
G{\aa}rding inequality. Then standard techniques in elliptic PDE theory
imply that $L$ is invertible up to some arbitrarily small perturbation to
$l_5$. Therefore, generically, it is reasonable to assume that $L$ is a
bounded isomorphism.

Let $D$ be the natural domain on which $L$ maps $W^k\rightarrow W^k$:
\begin{equation}
  \label{eq:D}
  D:=\{u\in V^k\cap V^*_k \,|\,
  (d^* +l_2)u\in V^{k-1},(d+l_1)u\in  V^*_{k+1}\}.
\end{equation}
Our main perturbed problem is: given $f\in W^k$, find $u\in D$ such that
\begin{equation}
  \label{eq:strong}
  Lu=f.
\end{equation}
We reformulate it in the mixed form: given $f\in W^k$ find
$(\sigma,u)\in V^{k-1}\times V^k$ satisfying
\begin{subequations}
  \label{eq:m}
  \begin{align}
    \label{eq:m1}
    &(\sigma, \tau) - (u, d\tau) - (l_2u,\tau) = 0,
    && \forall \tau\in V^{k-1}, \\
    \label{eq:m2}
    &((d+l_3)\sigma,v)+((d+l_1)u,dv)+(l_4du,v)+(l_5u,v)=(f,v),
    &&\forall v\in V^k.
  \end{align}
\end{subequations}
The first equation~\eqref{eq:m1} is equivalent to $u\in V^*_k$ and
$\sigma=(d^*+l_2)u$. The second equation~\eqref{eq:m1} is equivalent to
$(d+l_1)u\in V^*_{k+1}$ and
$d^*(d+l_1)u=f-(d+l_3)\sigma-l_4du-l_5u$. Hence, if $(\sigma,u)$
solves~\eqref{eq:m}, then $u$ solves~\eqref{eq:strong}. Therefore, it makes
sense to use this mixed formulation to solve our problem.

In addition to the assumption that $L$ is a bounded isomorphism, we need
more regularity assumptions on $l_i$ to ensure that $L^{-1}(W)\subset D$ so
that~\eqref{eq:strong} has a solution. One of our main tool for analyzing
the discretization is the duality argument, where the dual problem $L'z=g$
has to be solved as well. Here
$L':V^k\cap V^*_k\rightarrow (V^k\cap V^*_k)'$ is the dual of $L$. We
collection the conditions under which all these continuous problems are
well-posed in a theorem:
\begin{theorem}[Continuous well-posedness]
  \label{thm:wp}
  Let $(W,d)$ be a Hilbert complex with the compactness property. Suppose
  $L$ defined in equation~\eqref{eq:Lw} is a bounded isomorphism and
  \begin{equation}
    \label{eq:as}
    (d^* +l_2)L^{-1}(W)\subset V^{k-1},\qquad
    (d^* +l_3^*)(L')^{-1}(W)\subset V^{k-1}.
  \end{equation}
  Then both the perturbed problem~\ref{eq:strong} and its mixed
  formulation~\ref{eq:m} are well-posed.
\end{theorem}
The proof is given in section~\ref{sec:continuous_wp}. The regularity
assumption is very mild, without which the perturbed problem does not even
make sense.  The solution operator to the dual problem $L'z=g$ will be used
frequently, so we give it a name. Let
\begin{equation*}
  K=(L')^{-1}:(V^k\cap V^*_k)'\rightarrow V^k\cap V^*_k.
\end{equation*}

Our first discrete result is the following fundamental theorem on mixed
methods for problems perturbed by lower-order terms.
\begin{theorem}[Discrete stability]
  \label{thm:stability}
  Under the assumptions of Theorem~\ref{thm:wp}, suppose further that
  $\norm{(dd^*+d^*d)K}_{W\rightarrow W}$ is bounded and the following
  operators are compact $W\rightarrow W$:
  \begin{equation*}
    dl_3^*K, \quad
    (l_1^*d-l_2^*d^*)K.
  \end{equation*}
  Let $V^k_h$ be a sequence of dense subcomplexes of $V^k$ admitting
  $W$-bounded cochain projections. Then the Galerkin projection of the
  mixed system~\eqref{eq:m} using the pair $V^{k-1}_h\times V^k_h$ is
  stable in the sense that there exist positive constants $h_0,C_0$ such
  that for any $h\in (0,h_0]$, there exists a unique discrete solution
  $(\sigma_h,u_h)\in V^{k-1}_h\times V^k_h$ satisfying~\eqref{eq:m} for
  test functions in $V^{k-1}_h\times V^k_h$ and that
  $\norm{\sigma_h}_V+\norm{u_h}_V \leq C_0\norm{f}$.
\end{theorem}
The proof is nontrivial and is given in Section~\ref{sec:stability}.  For
the de~Rham complex, suppose all the lower-order terms are multiplication
by smooth coefficients and the domain has $(1+\epsilon)$-regularity that
$\norm{Kf}_{H^{1+\epsilon}}\lesssim \norm{f}_{L^2}$ for $\epsilon>0$. By
definition, $L'=(d+l_2^*)(d^*+l_3^*)+(d^*+l_1^*)d+d^*l_4^*+l_5^*$ and
$K=(L')^{-1}$, we have
\begin{equation*}
  (dd^*+d^*d)K = I - (l_2^*d^*+dl_3^*+l_1^*d+d^*l_4^*+l_5^*+l_2^*l_3^*)K
\end{equation*}
is bounded $L^2\rightarrow L^2$. The compactness assumptions are satisfied
due to the compact embedding of Sobolev space $H^s$ into $L^2$ for
$s>0$. This proves the statement on the stability of mixed discretization
of Hodge Laplacian in the introduction.

It is well-known that stability guarantees optimal error rates in the
energy norm:
\begin{corollary}
  Under the assumptions of Theorem~\ref{thm:stability}, if $h\le h_0$, the
  unique discrete solution $(\sigma_h,u_h)$ satisfies
  \begin{equation*}
    \norm{\sigma-\sigma_h}_V+\norm{u-u_h}_V \lesssim
    \norm{(I-\pi_h)\sigma}_V+\norm{(I-\pi_h)u}_V.
  \end{equation*}
\end{corollary}
As is common for mixed methods, the energy norm estimate is crude because
it couples errors of different variables. For example, for the unperturbed
Hodge Laplacian solved with the FEEC pair
$\mathcal{P}_{r+1}\Lambda^{k+1}\times \mathcal{P}_{r}\Lambda^k$, the
convergence rate in $\norm{\sigma}$ is in fact $h^2$ higher than that in
$\norm{du}$ but is lumped together with it. Our next finer discrete result
gives the decoupled $W$-norm estimates for each variable similar to
estimates~\eqref{eq:werr0} for the unperturbed problem. For that, we need
more assumptions. To simplify the bookkeeping, we define some approximation
quantities:
\begin{equation}
  \label{eq:approx}
  \begin{gathered}
    \delta=\max\{\delta_0, \norm{(I-P_h)l^*_3K},
    \norm{(I-P_h)(l_5^*-l_2^*l_3^*)K},
    \norm{(I-P_h)P_\B l_4^*K},\},\\
    \eta=\max\{\eta_0,\mu_0,\delta, \norm{(I-P_h)dl_3^*K},
    \norm{(I-P_h)(l_1^*d-l_2^*d^*)K}, \}, \\
    \alpha = \delta+\eta^2+\mu_0.
  \end{gathered}
\end{equation}
where $\eta_0,\delta_0,\mu_0$ are defined in equation~\eqref{eq:approx0}
and all the operator norms are in $\norm{\ph}_{W\rightarrow W}$. As,
before, due to the compactness assumptions and density, all
$\delta,\eta,\alpha\rightarrow 0$ as $h\rightarrow 0$.
\begin{theorem}
  \label{thm:improved_est}
  In addition to the assumptions of Theorem \ref{thm:stability}, assume that
  \begin{equation*}
    \norm{d(l_1^*d-l_2^*d^*+l_5^*-l_2^*l_3^*)K}_{W\rightarrow W}
  \end{equation*}
  is bounded, then we have the following improved error estimates:
  \begin{multline*}
    \norm{\sigma-\sigma_h}\lesssim E(\sigma)
    +(\eta+\chi_{45}\sqrt{\alpha})E(d\sigma)\\
    +(\chi_{24}+\chi_3\eta+\chi_5\sqrt{\eta})E(u)
    +(\chi_3\alpha+\chi_{45}\sqrt{\alpha})E(du),
  \end{multline*}
  \begin{multline*}
    \norm{d(\sigma-\sigma_h)}\lesssim E(d\sigma)
    +(\chi_3\alpha+\chi_4+\chi_5\eta)E(du)\\
    +(\chi_{45}+\chi_3\eta+\chi_2\chi_3)E(u)
    +(\chi_3+\chi_4\eta)E(\sigma),
  \end{multline*}
  \begin{equation*}
    \norm{u-u_h}\lesssim E(u)
    +\eta E(du)+\eta E(\sigma)
    +(\alpha+\chi_{45}\sqrt{\alpha})E(d\sigma),
  \end{equation*}
  \begin{equation*}
    \norm{d(u-u_h)}\lesssim E(du)+\eta E(d\sigma)
    +\chi_{1345}E(u)+(\chi_3+\chi_{145}\eta)E(\sigma),
  \end{equation*}
  where $\chi_{i\ldots j}$ denote the presence of lower-order terms. For
  example, $\chi_{125}=1$ if $l_1\neq 0$ or $l_2\neq 0$ or $l_5\neq 0$, and
  $\chi_{125}=0$ otherwise.
\end{theorem}
The proof is subtle and is given in Section~\ref{sec:improved_est}. 
\begin{corollary}
  Suppose a Hodge Laplacian problem satisfies full $2$-regularity:
  \begin{equation*}
    \norm{Kf}_{H^{s+2}}\lesssim \norm{f}_{H^s},\quad
    \text{for all $s\geq 0$},
  \end{equation*}
  (for example, on a smooth domain), then the error estimates for the
  discretization using FEEC elements are given by Table~\ref{tb:conv}.
\end{corollary}
\begin{proof}
  Following the discussion after equation~\eqref{eq:approx0}, we have
  $\eta_0=O(h)$, $\delta_0=O(h^{\min(2,r+1)})$, $\mu_0=O(h^{r+1})$, where $r$
  is the largest degree of complete polynomials in $V_h$. Using the best
  approximation estimates for FEEC elements, we see that $\eta=O(h)$ and
  $\delta=O(h^{\min(2,r+1)})$ as well. Plugging these and the best
  approximation estimates into Theorem~\ref{thm:improved_est}, we get the rates
  in Table~\ref{tb:conv}.
\end{proof}

\section{Well-posedness at the continuous level}
\label{sec:continuous_wp}
In this section, we establish well-posedness results for the continuous
problem and its mixed formulation.
\subsection{Well-posedness for the primal form}
First, we prove that the perturbed bounded operator is almost always an
isomorphism.
\begin{lemma}
  \label{lem:almost_good}
  Let $(W^k,d)$ be a Hilbert complex having the compactness property with
  domains $V^k$. Let $L$ be defined as in~\eqref{eq:Lw}. Then,
  $L+\lambda I$ has a bounded inverse for all $\lambda\in \sn{C}$ except at
  a discrete subset (so at most countable).
\end{lemma}
\begin{proof}
  Let $M=\max_{i}\norm{l_i}_{W\rightarrow W}$ and $\gamma = 4M^2+M+1/2$.
  Then $L+\gamma I$ is coercive on $V^k\cap V^*_k$:
  \begin{equation*}
    \ip{Lu,u}+\gamma (u,u) \geq (1/2)[(u,u)+(d^*u,d^*u)+(du,du)].
  \end{equation*}
  The compactness property ensures that
  $I:V^k\cap V^*_k \hookrightarrow (V^k\cap V^*_k)'$ is compact, which
  makes $I(L+\gamma I)^{-1}$ compact on $(V^k\cap V^*_k)'$.  Spectral
  theory then implies that $I+\mu I(L+\gamma I)^{-1}$ has a bounded inverse
  for all $\mu\in\sn{C}$ except at a discrete subset. Then composing with
  the bounded isomorphism $L+\gamma I$ on the right proves the claim.
\end{proof}
In particular, this shows that either $L$ is invertible or $L+\epsilon I$
is invertible for any small enough nonzero $\epsilon$.

Then, we prove the well-posedness of our main problem.
\begin{lemma}
  Suppose $L$ defined in equation~\eqref{eq:Lw} is a bounded isomorphism
  and $(d^* +l_2)L^{-1}(W)\subset V^{k-1}$. Then $L^{-1}(W)\subset D$,
  where $D$ is defined in~\eqref{eq:D}. In particular,
  problem~\eqref{eq:strong} has a unique solution.
\end{lemma}
\begin{proof}
  Since $L$ is already an isomorphism, we only need to show that
  $L^{-1}(W)\subset D$. For any $f\in W$, let $u=L^{-1}f$. Then by
  definition,
  \begin{equation*}
    ((d^*+l_2) u, (d^*+l_3^*) v) + ((d+l_1)u, dv) 
    + (l_4du ,v) + (l_5u, v) = (f,v), \qquad \forall v\in V^k\cap V_k^*.
  \end{equation*}
  By assumption, $(d^* +l_2)u\in V^{k-1}$. Thus, we have,
  \begin{equation*}
    ((d+l_1)u, dv) = 
    (f-(d+l_3)(d^*+l_2) u-l_4du-l_5u, v), \qquad \forall v\in V^k\cap V_k^*.    
  \end{equation*}
  There is no $d^*v$ in the above. By the density of $V^k\cap V_k^*$ in
  $W^k$, we conclude that the above holds for all $v\in V^k$. Hence,
  $(d+l_1)u\in V^*_{k+1}$. Thus $u\in D$ proves the claim.
\end{proof}
We then prove a similar result for the dual problem.  Let $D'$ be the
natural domain on which $L'$ maps $W^k\rightarrow W^k$:
\begin{equation*}
  D':=\{u\in V^k\cap V^*_k \,|\,
  (d^* +l_3^*)u\in V^{k-1},(d+l_4^*)u\in  V^*_{k+1}\}.
\end{equation*}
Using the same argument, we get,
\begin{lemma}
  Suppose $L$ defined in equation~\eqref{eq:Lw} is a bounded isomorphism
  and $(d^* +l_3^*)(L')^{-1}(W)\subset V^{k-1}$. Then
  $(L')^{-1}(W)\subset D'$.
\end{lemma}
\subsection{Well-posedness for the mixed form}
We then turn to mixed system~\eqref{eq:m}.  Its associated bilinear form
$B:(V^{k-1}\times V^k)^2\rightarrow \sn{R}$ is:
\begin{multline}
  \label{eq:bf}
  B((\sigma,u),(\tau,v)):=(\sigma,\tau)+(du,dv)+
  (d\sigma,v)-(u,d\tau)\\
  -(l_2u,\tau)+(l_3\sigma,v)+(l_1u,dv)+(l_4du,v)+(l_5u,v).
\end{multline}
We call this bilinear form \emph{well-posed} if and only for any
$(g,f)\in V'_{k-1}\times V'_k$, there exists a unique solution
$(\sigma,u)\in V^{k-1}\times V^k$ satisfying:
\begin{equation*}
  B((\sigma,u),(\tau,v)) = \ip{g,\tau}_{V'\times V}
  + \ip{f,v}_{V'\times V}, \qquad
  \forall (\tau,v) \in V^{k-1}\times V^k.
\end{equation*}
From the discussion after equation~\ref{eq:m}, we see that the
well-posedness of $B$ implies that $L^{-1}(W)\in D$. But it also implies
the well-posedness of the dual mixed problem: given any
$(g,f)\in V'_{k-1}\times V'_k$, find $(\xi, z)\in V^{k-1}\times V^k$
satisfying
\begin{equation*}
  B((\rho,w),(\xi,z)) = \ip{g,\rho}_{V'\times V}
  + \ip{f,w}_{V'\times V}, \qquad
  \forall (\rho,w) \in V^{k-1}\times V^k.  
\end{equation*}
A similar argument shows that the well-posedness of $B$ implies
$(L')^{-1}(W)\in D'$ as well.  We collect these results in a lemma:
\begin{lemma}
  Suppose $B$ defined in~\eqref{eq:bf} is well-posed and let $L$ be defined
  as in~\eqref{eq:Lw}. Then $L$ is a bounded isomorphism,
  $L^{-1}(W)\subset D$, and $(L')^{-1}(W)\subset D'$.
\end{lemma}
Moreover, the converse is also true.
\begin{lemma}
  \label{lem:mwp}
  Suppose $L$ defined in~\eqref{eq:Lw} is a bounded isomorphism. Then $B$
  defined in~\eqref{eq:bf} is well-posed if and only if
  condition~\eqref{eq:as} holds.
\end{lemma}
\begin{proof}
  The only if part is clear from the previous lemmas. We only need to show
  the if part.

  First, we show that $B$ satisfies a G{\aa}rding-like inequality: there
  exist positive constants $a,b,c$ depending only on
  $\norm{l_i}_{W\rightarrow W}$ such that
  \begin{equation}
    \label{eq:garding}
    B((\sigma,u),(\sigma,u+ad\sigma)) 
    \geq b(\norm{\sigma}_V+\norm{u}_V)^2-c\norm{u}^2,
    \qquad
    \forall (\sigma,u)\in V^{k-1}\times V^k.
  \end{equation}
  Direct computation using Cauchy-Schwarz inequality shows that there exist
  constants $c_1,c_2$ depending only on $\norm{l_i}_{W\rightarrow W}$ such
  that
  \begin{gather*}
    B((\sigma,u),(\sigma,u))\geq 
    (1/2)(\norm{\sigma}^2+\norm{du}^2)-c_1\norm{u}^2,\\
    B((\sigma,u),(0,d\sigma))\geq 
    (1/2)\norm{d\sigma}^2-c_2(\norm{\sigma}^2+\norm{du}^2+\norm{u}^2).
  \end{gather*}
  Multiplying the second inequality by any positive $a<1/(2c_2)$ and adding it
  to the first inequality, we get the claim.
  
  Second, fix any $(\sigma,u)\in V^{k-1}\times V^k$. We solve a dual
  problem using $u$ as data: let $z=cKu$ and $\xi=-(d^*+l_3^*)z$. By
  assumption, $\xi\in V^{k-1}$. Direct computation shows that
  \begin{equation*}
    B((\rho, w), (\xi,z)) = (cu, w), \qquad \forall 
    (\rho, w)\in V^{k-1}\times V^k.
  \end{equation*}

  Finally, we add $(\xi,z)$ to our choice of test functions in the first
  step and get
  \begin{equation*}
    B((\sigma,u),(\sigma+\xi,u+ad\sigma+z)) 
    \geq b(\norm{\sigma}_V+\norm{u}_V)^2.
  \end{equation*}
  Further, from the definition of $(\xi, z)$, we have
  \begin{equation*}
    \norm{\sigma+\xi}_V+\norm{u+ad\sigma+z}_V
    \leq M\norm{\sigma}_V+\norm{u}_V,
  \end{equation*}
  where the constant $M$ depends only on $a$,
  $\norm{d(d^*+l_3^*)K}_{W\rightarrow W}$, and
  $\norm{dK}_{W\rightarrow W}$.  Thus $B$ satisfies the inf-sup
  condition. Similarly, for any fixed nontrivial
  $(\tau, v)\in V^{k-1}\times V^k$, we let $u=L^{-1}v$ and
  $\sigma = (d^*+l_2)u$. By assumption, $\sigma\in V^{k-1}$. Direct
  computation shows that
  \begin{equation*}
    B((\sigma, u),(\tau, v))=(v,v)>0.
  \end{equation*}
  It is well-known that these two conditions imply that $B$ is
  well-posed~\cite{Babuska1970}.
\end{proof}
Given these lemmas, Theorem~\ref{thm:wp} is clearly true.

\section{Discrete Stability through new projections}
\label{sec:stability}
In this section, we prove Theorem~\ref{thm:stability}. The idea of the
proof is similar to the that of the if part of Lemma~\ref{lem:mwp}. First,
due to the subcomplex property, the estimate~\eqref{eq:garding} still holds
at the discrete level with the same constants $a,b,c$. Second, fix any
$(\sigma,u)\in V^{k-1}_h\times V^k_h$. We can again solve a dual problem
using $u$ as data: let $z=cKu$ and $\xi=-(d^*+l_3^*)z$. Then we have
$\xi\in V^{k-1}$, $\norm{\xi}_V+\norm{z}_V\lesssim \norm{u}$ independent of
$h$, and
\begin{equation*}
  B((\rho, w), (\xi,z)) = (cu, w), \qquad \forall 
  (\rho, w)\in V^{k-1}\times V^k.
\end{equation*}
But we can no longer add $(\xi,z)$ to our choice of test functions because
$(\xi,z)$ is not discrete. In the rest of this section, we construct a
discrete pair $(\xi_h,z_h)\in V^{k-1}_h\times V^k_h$ such that
\begin{equation}
  \label{eq:dds}
  \begin{gathered}
    \norm{\xi_h}_V+\norm{z_h}_V\lesssim \norm{\xi}_V+\norm{z}_V
    \text{ uniformly in $h$, and}\\
    |B((\rho,w),(\xi-\xi_h,z-z_h))| \leq
    \epsilon_h(\norm{\rho}_V+\norm{w}_V)(\norm{\xi}_V+\norm{z}_V),
  \end{gathered}
\end{equation}
for all $(\rho, w)\in V^{k-1}_h\times V^k_h$, where
$\epsilon_h\rightarrow 0$ as $h\rightarrow 0$. Given such a pair, we can
add it to our choice of test functions:
\begin{equation*}
  B((\sigma,u),(\sigma+\xi_h,u+ad\sigma+z_h))
  \geq (b-c\epsilon_h)(\norm{\sigma}_V+\norm{u}_V)^2.
\end{equation*}
Further, there also exists $M>0$ bounded uniformly in $h$, such that
\begin{equation*}
  \norm{\sigma+\xi_h}_V+\norm{u+ad\sigma+z_h}_V
  \leq M(\norm{\sigma}_V+\norm{u}_V).
\end{equation*}
Choose a sufficiently small $h_0$ such that $\epsilon_{h_0}<b/c$.  Then for
all $h<h_0$, the bilinear form $B((\sigma,u),(\tau,v))$ satisfies the
inf-sup condition on $V^{k-1}_h\times V^k_h$ with the inf-sup constant
bounded uniformly below by $(b-c\epsilon_{h_0})/M$. Since
$V^{k-1}_h\times V^k_h$ is of finite dimension, this establishes the
well-posedness. Thus Theorem~\ref{thm:stability} is proved.

An obvious choice for $(\xi_h,z_h)$ in~\eqref{eq:dds} is the elliptic
projection given by
\begin{equation*}
  B((\rho,w),(\xi_h,z_h))=B((\rho,w),(\xi,z)), \qquad
  \forall (\rho, w)\in V^{k-1}_h\times V^k_h.
\end{equation*}
Then $\epsilon_h=0$. But since we have not proved the well-posedness of $B$
on the discrete level, we neither know a discrete solution exists nor can
we show the uniform estimates. The next most obvious choice is obtained
using the elliptic projection of the unperturbed problem:
\begin{equation*}
  B_0((\rho,w,p),(\xi_h,z_h,q_h))=B_0((\rho,w),(\xi,z)), \qquad
  \forall (\rho, w,p)\in V^{k-1}_h\times V^k_h\times \h^k_h.
\end{equation*}
Then we have the existence and uniform bounds. But the second estimate
in~\eqref{eq:dds} fails. In what follows, we develop two new projection
operators to correct the elliptic projection for the unperturbed problem so
that both conditions in~\eqref{eq:dds} holds. In fact, we do a lot
more. Our elaborately chosen $(\xi_h,z_h)$ will not only
satisfy~\eqref{eq:dds}, but also have explicit and optimal error rates in
quantities like $\norm{\xi-\xi_h}, \norm{z-z_h}$, and
$|B((\rho,w),(\xi-\xi_h,z-z_h)|$. This is made precise in
Theorem~\ref{thm:ddt}. This result also form the basis of the improved
error estimates later. Moreover, the two new projection operators enjoy
many properties making them interesting in their own right.

\subsection{Generalized Canonical Projection}
Thus far, we have three projections in FEEC: the orthogonal projection
$P_h$, the cochain projection $\pi_h$ which commutes with $d$ but has no
orthogonality property, and the elliptic projection $K_{0h}P_hL_0$ which
misses the harmonic part and is only well-defined on the subspace
$D(L_0)$. Here, we introduce a new projection operator
$\Pi_h:V^k\rightarrow V^k_h$ given by
\begin{equation*}
  \Pi_h:=P_{\z_h}+d^*_hK_{0h}P_hd.
\end{equation*}
In the above and for the rest of this paper, for any subspace $X$ of $W$, we
use the notation $P_X:W\rightarrow X$ for the $W$-orthogonal projection. Among
other properties, this $\Pi_h$ satisfies a commutative property generalizing
that of the canonical projection for classical elements like Raviart--Thomas.
\begin{theorem}
  \label{thm:Pih}
  Suppose $(W,d)$ is a Hilbert complex satisfying the compactness property
  and $V^k_h$ are dense discrete subcomplexes admitting $W$-bounded cochain
  projections. Then $\Pi_h$ is a projection uniformly bounded in the
  $V$-norm. Further $d\Pi_h=P_{\B_h}d$. Let $\eta_0,\alpha_0$ be defined as
  in equation~\eqref{eq:approx0}.  Then, for any $w\in V$,
  \begin{equation*}
    \norm{\Pi_hw}\lesssim \norm{w}+\eta_0\norm{dw}, \qquad
    \norm{(I-\Pi_h)w} \lesssim \norm{(I-\pi_h)w}+\eta_0 \norm{dw}.
  \end{equation*}
  Moreover, it satisfies ``partial orthogonality'': for any $w,v\in V$,
  \begin{equation*}
    |((I-\Pi_h)w,v)|\lesssim 
    (\norm{(I-\pi_h)w}+\eta_0\norm{dw})(\norm{(I-\pi_h)v}+\eta_0\norm{dv})
    +\alpha_0\norm{dv}\norm{dw}.
  \end{equation*}
\end{theorem}
\begin{proof}
  The stability of the unperturbed discrete problem~\eqref{eq:m0} implies
  that $\Pi_h$ is uniformly bounded in the $V$-norm. By subcomplex
  property, we have $P_hdv = dv$ for $v\in V_h$. Thus,
  \begin{equation*}
    \Pi_h v = P_{\z_h}v+d^*_hK_{0h}dv=P_{\z_h}v+P_{\B^*_h}v = v,
  \end{equation*}
  showing that $\Pi_h$ is a projection. We know from
  equation~\eqref{eq:hodgeh} that for the unperturbed problem
  $dd_h^*K_{0h}=P_{\B_h}$. This proves that
  $d\Pi_h=dd_h^*K_{0h}P_hd=P_{\B_h}d$. We then prove the first two
  estimates.  Fix any $w\in V$. We split
  $w-\Pi_hw=(P_\z-P_{\z_h})w + (P_{\B^*}w-P_{\B^*_h}\Pi_hw)$. The second
  term can be bounded using the error estimates \eqref{eq:est-op} for
  $K_{0h}$:
  \begin{equation*}
    \norm{P_{\B^*}w-P_{\B^*_h}\Pi_hw} = \norm{d^*K_{0}dw-d^*_hK_{0h}P_hdw}
    \lesssim \eta_0 \norm{dw}.
  \end{equation*}
  We then deal with the first term. The subcomplex property ensures
  $\z_h\subset \z$ and the cochain property of $\pi_h$ ensures
  $\pi_h\z\subset \z_h$. These two lead to
  \begin{equation*}
    \norm{(P_\z-P_{\z_h})w} \lesssim \norm{(I-\pi_h)P_{\z}w}.
  \end{equation*}
  But $(I-\pi_h)P_{\z}w=(I-\pi_h)w-(I-\pi_h)P_{\B^*}w$ and
  $P_{\B^*}=d^*K_0d$. Thus,
  \begin{equation*}
    \norm{(P_\z-P_{\z_h})w}
    \lesssim \norm{(I-\pi_h)w}+\norm{(I-\pi_h)d^*K_0dw}
    \lesssim \norm{(I-\pi_h)w}+\eta_0\norm{dw}.
  \end{equation*}
  Combining the estimates for the two parts, we get,
  \begin{equation*}
    \norm{(I-\Pi_h)w} = 
    \norm{(P_\z-P_{\z_h})w + (P_{\B^*}w-P_{\B^*_h}\Pi_hw)}
    \lesssim 
    \norm{(I-\pi_h)w}+\eta_0\norm{dw}.
  \end{equation*}
  By triangle inequality, we get
  \begin{equation*}
    \norm{\Pi_hw} 
    \leq \norm{w} + \norm{(I-\Pi_h)w}
    \lesssim \norm{w}+\eta_0\norm{dw}
  \end{equation*}
  as well. This proves the first two estimates.  Finally, for any
  $w,v\in V$, we have
  \begin{multline*}
    ((I-\Pi_h)w, v)=((P_\z-P_{\z_h})w,v) 
    +((P_{\B^*}-P_{\B^*_h}\Pi_h)w,(I-\pi_h)v)\\
    +((P_{\B^*}-P_{\B^*_h}\Pi_h)w,\pi_hv).
  \end{multline*}
  The first two terms can be bounded as before. The last term is bounded by
  \begin{equation*}
    (d^*K_{0}dw-d^*_hK_{0h}P_hdw, \pi_hv)
    =((K_{0}-K_{0h}P_h)dw, \pi_hdv)
    \lesssim\alpha_0\norm{dw}\norm{dv},
  \end{equation*}
  where the error estimate \eqref{eq:est-op} is used again. This finishes
  the proof.
\end{proof}

\subsection{Modified elliptic projection}
We modify the unperturbed elliptic projection slightly to accommodate the
harmonic forms.
\begin{theorem}
  \label{thm:mep}
  For any $z\in D(L_0)$, let $z_h=K_{0h}P_h L_0z+P_{\h_h}P_{\h}z$.  Then,
  \begin{gather*}
    \norm{z-z_h}\lesssim \alpha_0\norm{L_0z}, 
    \qquad 
    \norm{d(z-z_h)}+\norm{d^*z-d^*_hz_h}\lesssim \eta_0\norm{L_0z},\\
    \norm{P_h(d^*dz+dd^*z)-(d^*_hdz_h+dd^*_hz_h)}\leq \mu_0\norm{L_0z}.
  \end{gather*}
\end{theorem}
\begin{proof}
  By equation~\eqref{eq:hodge}, we have $P_{\h^\perp}z=K_0L_0z$. Thus we
  have the splitting
  \begin{equation*}
    z-z_h=(P_{\h^\perp}z-P_{\h^\perp_h}z_h) + (P_{\h}z-P_{\h_h}z_h)=
    (K_0-K_{0h}P_h)L_0z+(I-P_{\h_h})P_{\h}z.
  \end{equation*}
  The first term has been estimated by \eqref{eq:est-op}. For the second
  term, since $P_{\h_h}P_{\h}=P_{\z_h}P_{\h}$ and $\pi_h\z\subset\z_h$, we
  have
  \begin{equation*}
    \norm{(I-P_{\h_h})P_{\h}z}_W\leq \norm{(I-\pi_h)P_{\h}z}_W
    \leq \norm{(I-\pi_h)P_{\h}}_{W\rightarrow W}\norm{z}= \mu_0\norm{z},
  \end{equation*}
  which proves the first estimate.  The second estimate follows from
  \eqref{eq:est-op} directly. Finally for the last estimate, we use the
  continuous and discrete Hodge decomposition \eqref{eq:hodge}
  \eqref{eq:hodgeh}, we get
  \begin{equation*}
    (d^*_hd+dd^*_h)z_h = 
    (d^*_hd+dd^*_h)K_{0h}P_hL_0z = (P_{\B_h}+P_{\B^*_h})P_hL_0z.
  \end{equation*}
  Moreover, by definition, $P_{\h}L_0z=0$. Thus,
  \begin{equation*}
    P_h(d^*dz+dd^*z-d^*_hdz_h-dd^*_hz_h)
    =P_hL_0z-(P_{\B_h}+P_{\B^*_h})P_hL_0z=P_{\h_h}L_0z.
  \end{equation*}
  The right-hand side is just $\norm{p-p_0}$ for the unperturbed problem
  with $L_0z$ as data. By \eqref{eq:werr0},
  \begin{equation*}
    \norm{P_{\h_h}L_0z} = \norm{(P_{\h_h}-P_{\h})L_0z} 
    \lesssim 0+\mu_0\norm{P_{\B}L_0z}\leq \mu_0\norm{L_0z},
  \end{equation*}
  which proves the last estimate.
\end{proof}

\subsection{Projection of the dual solution}
We are now ready to construct the discrete pair $(\xi_h,z_h)$ satisfying
the conditions~\eqref{eq:dds} in the proof of the discrete stability
theorem. In fact, we prove a stronger result where the first variable
$\rho$ is allowed to be in $V$ instead of $V_h$ and derive explicit error
estimates.
\begin{theorem}
  \label{thm:ddt}
  Under the assumption of Theorem~\ref{thm:stability}, for any $g \in W^k$,
  let $z=Kg$, $\xi=-(d^*+l_3^*)z$, $z_h=K_{0h}P_hL_0z+P_{\h_h}P_{\h}z$, and
  $\xi_h=-d^*_hz_h-\Pi_hl_3^*z$. Then,
  \begin{equation*}
    \norm{z-z_h}\lesssim \alpha_0\norm{g}, \qquad 
    \norm{d(z-z_h)} \lesssim \eta_0 \norm{g},
    \qquad \norm{\xi-\xi_h}\lesssim \eta\norm{g}.
  \end{equation*}
  Further, for any $(\rho,w)\in V^{k+1}\times V^k_h$, we have,
  \begin{equation*}
    |B((\rho,w),(\xi-\xi_h,z-z_h))| \lesssim
    [\eta\norm{\rho}
    +\alpha\norm{d\rho}
    +(\mu_0+\chi_{123}\eta+\chi_5\alpha)\norm{w}
    +\chi_4\alpha\norm{dw}]\norm{g}.
  \end{equation*}
\end{theorem}
\begin{proof}
  Using the regularity assumption that $\norm{L_0K}_{W\rightarrow W}$ is
  bounded, the estimate for $\norm{z-z_h}$ and $\norm{d(z-z_h)}$ follows
  directly from Theorem~\ref{thm:mep}. From the same theorem, for $\xi$, we
  have
  \begin{equation*}
    \norm{\xi-\xi_h}\leq \norm{d^*z-d^*_hz_h}+\norm{(I-\Pi_h)l_3^*z}
    \lesssim \eta_0\norm{g} +\norm{(I-\Pi_h)l_3^*Kg}.
  \end{equation*}
  For the second term, using quantities defined in \eqref{eq:approx}, we
  have:
  \begin{equation*}
    \norm{(I-\Pi_h)l_3^*Kg}
    \lesssim \norm{(I-\pi_h)l_3^*Kg}+\eta\norm{dl_3^*Kg}
    \lesssim \eta\norm{g}.
  \end{equation*}
  The last estimate is just a direct computation using the error estimates
  in Theorem~\ref{thm:mep}, quantities defined in \eqref{eq:approx}, and
  the Cauchy-Schwarz inequality.
\end{proof}
For the proof of Theorem~\ref{thm:stability} at the beginning of this
section, we get $(\xi_h,z_h)$ by applying this theorem to
$g=cu\in V_h\subset W$. We note that $(\xi,z)$ here is the same as the one
defined there. We check that condition~\eqref{eq:dds} is satisfied. First,
\begin{gather*}
  \norm{z_h}_V=\norm{K_{0h}P_hL_0z+P_{\h_h}P_{\h}z}_V\lesssim \norm{z}_V,\\
  \norm{\xi_h}_V=\norm{-d^*_hK_{0h}P_hL_0z-\Pi_hl_3^*z}_V\lesssim
  \norm{d\xi}+\norm{z}_V\lesssim \norm{\xi}_V+\norm{z}_V,
\end{gather*}
where the constants depend only on the stability constant of the continuous
and discrete unperturbed problem, which is either independent of $h$ or
bounded uniformly in $h$. Second, as mentioned before, the compactness
assumptions and density, $\alpha, \eta, \mu_0\rightarrow 0$ as
$h\rightarrow 0$. Thus condition~\eqref{eq:dds} is verified. This finishes
the proof of Theorem~\ref{thm:stability}.

\section{Proof of Improved Error Estimates}
\label{sec:improved_est}
In this section, we prove Theorem~\ref{thm:improved_est}. To make the
notation more compact, we use $e_u:=u-u_h$ and
$E_u=(I-\pi_h)u$. Corresponding quantities for $\sigma$ are similarly
defined.  The Galerkin orthogonality equation reads:
\begin{subequations}
  \label{eq:ge}
  \begin{align}
    \label{eq:ge1}&
    (e_\sigma, \tau) - (e_u, d\tau) - (l_2e_u,\tau) = 0,
    &&\forall \tau\in V_h^{k-1}, \\
    \label{eq:ge2}&
    ((d+l_3)e_\sigma,v)+((d+l_1)e_u,dv)
    +(l_4de_u,v)+(l_5e_u,v)=0,&&\forall v\in V_h^k.
  \end{align}
\end{subequations}

\subsection{Preliminary estimates for 
$\norm{de_\sigma}$ and $\norm{de_u}$}
Optimal estimates for these two terms can be obtained directly from the
error equations \eqref{eq:ge} with carefully chosen test functions.
\begin{lemma}
  \label{lem:pdsigma}
  For any $(\sigma,u)$ solving \eqref{eq:m} and $(\sigma_h,u_h)$ solving
  its Galerkin projection,
  \begin{equation*}
    \norm{de_\sigma}\lesssim 
    \norm{dE_\sigma}+\chi_3\norm{e_\sigma}
    +\chi_4 \norm{de_u}+\chi_5\norm{e_u}.
  \end{equation*}
\end{lemma}
\begin{proof}
  Restricting the test function space to $\B_h$ in equation \eqref{eq:ge2}
  leads to:
  \begin{equation}
    \label{eq:pbhdsigma}
    P_{\B_h}(de_\sigma+l_3e_\sigma+l_4de_u+l_5e_u)=0.
  \end{equation}
  Thus,
  $de_\sigma=(I-P_{\B_h})de_\sigma+P_{\B_h}de_\sigma
  =(I-P_{\B_h})de_\sigma-P_{\B_h}(l_3e_\sigma+l_4de_u+l_5e_u)$.
  Because $\pi_h$ maps to $\B$ to $\B_h$, we have
  $\norm{(I-P_{\B_h})de_\sigma}\lesssim
  \norm{(I-\pi_h)de_\sigma}=\norm{dE_\sigma}$ proving the claim.
\end{proof}
\begin{lemma}
  \label{lem:pdu}
  For any $(\sigma,u)$ solving \eqref{eq:m} and $(\sigma_h,u_h)$ solving
  its Galerkin projection,
  \begin{equation*}
    \norm{de_u}\lesssim \norm{dE_u}+\eta\norm{de_\sigma}
    +\chi_{145}\norm{e_u}+\chi_3\norm{e_\sigma}.
  \end{equation*}
\end{lemma}
\begin{proof}
  Let $v_h=P_{\B_h^*}(\pi_hu-u_h)$ in the second error equation
  \eqref{eq:ge2}. We have
  \begin{equation}
    \label{eq:pdum}
    (de_u,dv_h) = 
    -(de_\sigma,v_h)-[(l_1e_u,dv_h)+(l_3e_\sigma+l_5e_u,v_h)]-(l_4de_u,v_h).
  \end{equation}
  By discrete Poincar\'e inequality $\norm{v_h}\lesssim \norm{dv_h}$ for
  $v_h\in \B_h^*$, the second term in \eqref{eq:pdum} becomes
  \begin{equation*}
    |(l_1e_u,dv_h)+(l_3e_\sigma+l_5e_u,v_h)|
    \lesssim (\chi_{15}\norm{e_u}+\chi_3\norm{e_\sigma})\norm{dv_h}.
  \end{equation*}
  Because $de_u=d(u-\pi_hu+\pi_hu-u_h)=dE_u+dv_h$, the last term in
  \eqref{eq:pdum} satisfies
  \begin{equation*}
    |(l_4de_u, v_h)| = |(l_4dE_u,v_h)+(l_4dv_h,v_h)| 
    \lesssim (\norm{dE_u}+\norm{v_h})\norm{dv_h}.
  \end{equation*}
  We finally estimate the first term in \eqref{eq:pdum}.  Let
  $v=P_{\B^*}v_h$.  Then $d(\pi_hv-v_h) = 0$ implies $\pi_hv-v_h\in \z_h$,
  so $(v_h-v)\perp (\pi_h v-v_h)$.  Thus,
  \begin{equation*}
    \norm{v-v_h}\leq \norm{(I-\pi_h)v}
    =\norm{(I-\pi_h)d^*Kd v_h} \lesssim \eta\norm{dv_h}.
  \end{equation*}
  This implies,
  \begin{equation*}
    |(de_\sigma, v_h)| = |(de_\sigma, v-v_h)|
    \lesssim \eta\norm{de_\sigma}\norm{dv_h}.
  \end{equation*}
  Combining all these estimates and
  $\norm{de_u}\leq \norm{dE_u}+\norm{dv_h}$ gives the estimate in the
  claim.
\end{proof}

\subsection{Duality lemma}
The optimal $W$-norm estimates for $\norm{e_u}$ and $\norm{e_\sigma}$
require more work.
\begin{lemma}
  \label{lem:duality}
  Under the assumptions of Theorem~\ref{thm:improved_est}, for any
  $g\in W$,
  \begin{equation*}
    |(\Pi_h e_u,g)|\lesssim
    [\eta\norm{e_\sigma} + \alpha\norm{de_\sigma}
    +(\mu_0+\chi_{123}\eta+\chi_5\alpha)\norm{e_u}
    +(\mu_0\eta+\chi_{12345}\alpha)\norm{de_u}]\norm{g}.
  \end{equation*}
\end{lemma}
\begin{proof}
  Given $g$, let $(\xi,z)$ and $(\xi_h,z_h)$ be defined as in
  Theorem~\ref{thm:ddt}. We have,
  \begin{equation*}
    (\Pi_h e_u, g) = B((e_\sigma,\Pi_h e_u),(\xi,z))
    =B((e_\sigma,\Pi_h e_u),(\xi_h,z_h))-B((e_\sigma,\Pi_h e_u),(e_\xi,e_z)).
  \end{equation*}
  Galerkin orthogonality~\eqref{eq:ge} states that
  $B((e_\sigma,e_u),(\tau,v))=0$ for all discrete $(\tau,v)$. Thus,
  \begin{equation*}
    (\Pi_h e_u, g)=-B((0,(I-\Pi_h) e_u),(\xi_h,z_h))
    -B((e_\sigma,\Pi_h e_u),(e_\xi,e_z)).
  \end{equation*}
  The second term above can be estimated by the last inequality in
  Theorem~\ref{thm:ddt} and
  \begin{equation*}
    \norm{\Pi_he_u}\lesssim \norm{e_u}+\eta\norm{de_u}, \qquad
    \norm{d\Pi_he_u}\lesssim \norm{de_u}.
  \end{equation*}
  The result is the following bound:
  \begin{multline*}
    |B((e_\sigma,\Pi_h e_u),(e_\xi,e_z))|\lesssim
    [\eta\norm{e_\sigma} + \alpha\norm{de_\sigma}
    +(\mu_0+\chi_{123}\eta+\chi_5\alpha)\norm{e_u}\\
    +(\mu_0\eta+\chi_4\alpha)\norm{de_u}]\norm{g}.
  \end{multline*}
  We bound the first term before by splitting it into three parts:
  \begin{multline*}
    |B((0,(I-\Pi_h)e_u),(\xi_h,z_h))|
    \leq
    |((I-\Pi_h)e_u,d \xi_h)+(d(I-\Pi_h)e_u,d z_h)|
    \\
    +|(l_2(I-\Pi_h)e_u,\xi_h) 
     +(l_1(I-\Pi_h)e_u,d z_h)
     +(l_5(I-\Pi_h)e_u, z_h)|\\
    +|(l_4d(I-\Pi_h)e_u,z_h)|
    =:Q_1+Q_2+Q_3.
  \end{multline*}
  We have $Q_1=0$ because $P_{\B_h}\Pi_h=P_{\B_h}$ and $d\Pi_h=P_{\B_h}d$.
  The second term is:
  \begin{equation*}
    Q_2\leq
    |((I-\Pi_h)e_u,l_2^*\xi+l_1^*dz+l_5^*z)|
    +|((I-\Pi_h)e_u,l_2^*e_\xi+l_1^*de_z+l_5^*e_z)|.
  \end{equation*}
  We note that
  $l_2^*\xi+l_1^*dz+l_5^*z=[-l_2^*d^*+l_1^*d+(l_5^*-l_2^*l_3^*)]Kg$. The
  first term above can be bounded by the regularity assumptions and the
  last estimate in Theorem~\ref{thm:Pih}. The second term can be bounded
  using estimates in Theorem~\ref{thm:ddt}. The result is:
  \begin{equation*}
    Q_2\lesssim
    [(\chi_{12}\eta+\chi_2\delta+\chi_5\alpha)\norm{e_u}+
     \chi_{125}\alpha\norm{de_u}]\norm{g}.
  \end{equation*}
  For $Q_3$, we have,
  \begin{equation*}
    Q_3 = |((I-P_{\B_h})de_u,l_4^*z_h)|
    \leq |((I-P_{\B_h})de_u,l_4^*z)|
    +|((I-P_{\B_h})de_u,l_4^*e_z)|.
  \end{equation*}
  For the first term
  $((I-P_{\B_h})de_u,l_4^*z)=(de_u, (P_{\B}-P_{\B_h})l_4^*z)$ can be
  bounded using~\ref{eq:approx}. The second term can be bounded by
  $|(de_u,l_4^*e_z)|$ and Theorem~\ref{thm:ddt}. The final result is,
  \begin{equation*}
    Q_3 \lesssim \chi_4\alpha\norm{de_u}\norm{g}.
  \end{equation*}
  Combining the all estimates together, we get the estimate in the claim.
\end{proof}

\subsection{Preliminary estimate for $\norm{e_u}$}
\begin{lemma}
  \label{lem:pu}
  Under the assumptions of Theorem~\ref{thm:improved_est}, we have
  \begin{equation*}
    \norm{e_u}\lesssim 
    \norm{E_u}+\eta\norm{e_\sigma}+\alpha\norm{de_\sigma}
    +\eta\norm{de_u}.
  \end{equation*}
\end{lemma}
\begin{proof}
  Let $g=\Pi_he_u$ in Lemma~\ref{lem:duality} we get an estimate for
  $(\Pi_he_u, \Pi_he_u)$. Then,
  \begin{equation*}
    \norm{e_u}\leq\norm{\Pi_he_u}+\norm{(I-\Pi_h)e_u}
    \lesssim\norm{\Pi_he_u}+\norm{E_u}+\eta\norm{de_u}.
  \end{equation*}
  For sufficiently small $h$, we hide the $\norm{e_u}$ term on the right in
  the left-hand side. The result is the estimate in the claim.
\end{proof}

\subsection{Preliminary estimate for $\norm{e_\sigma}$}
\begin{lemma}
  \label{lem:psigma}
  Under the assumptions of Theorem~\ref{thm:improved_est}, 
  \begin{equation*}
    \norm{e_\sigma}\lesssim
    \norm{E_\sigma} + (\eta+\chi_{45}\sqrt{\alpha})\norm{de_\sigma}
    +(\chi_{24}+\chi_3\eta+\chi_5\sqrt{\eta})\norm{e_u}
    +(\chi_3\alpha+\chi_{45}\sqrt{\alpha})\norm{de_u}.
  \end{equation*}  
\end{lemma}
\begin{proof}
  Equation \eqref{eq:ge1} implies:
  \begin{multline*}
    (e_\sigma, e_\sigma)
    =(e_\sigma, (I-\Pi_h)e_\sigma)+(e_\sigma, \Pi_h e_\sigma) \\
    =(e_\sigma, (I-\Pi_h)e_\sigma)
    +(l_2e_u, \Pi_h e_\sigma)
    +(e_u,d\Pi_he_\sigma).
  \end{multline*}
  The first term above is bounded by:
  \begin{equation*}
    |(e_\sigma, (I-\Pi_h)e_\sigma)|
    \leq\norm{e_\sigma}\norm{(I-\Pi_h)\sigma}
    \lesssim (\norm{E_\sigma}+\eta\norm{dE_\sigma})\norm{e_\sigma}.
  \end{equation*}
  The second term is bounded using
  $\norm{\Pi_he_\sigma}\lesssim \norm{e_\sigma}+\eta\norm{de_\sigma}$ from
  theorem \ref{thm:Pih},
  \begin{equation*}
    |(l_2e_u, \Pi_h e_\sigma)|
    \lesssim \chi_2\norm{e_u}(\norm{e_\sigma}+\eta\norm{de_\sigma})
  \end{equation*}
  The last term is estimated by the duality lemma.
  Because $P_{\B_h}\Pi_h=P_{\B_h}$, we have,
  \begin{equation*}
    (e_u,d\Pi_he_\sigma)
    =(e_u,P_{\B_h}de_\sigma)
    =(P_{\B_h}e_u,P_{\B_h}de_\sigma)
    =(\Pi_he_u,P_{\B_h}de_\sigma).
  \end{equation*}
  We apply \ref{lem:duality} with $g=P_{\B_h}de_\sigma$ and get
  \begin{multline*}
    |(\Pi_h e_u,g)|\lesssim
    [\eta\norm{e_\sigma} + \alpha\norm{de_\sigma}
    +(\mu_0+\chi_{123}\eta+\chi_5\alpha)\norm{e_u}\\
    +(\mu_0\eta+\chi_{12345}\alpha)\norm{de_u}]\norm{P_{\B_h}de_\sigma}.
  \end{multline*}
  From equation~\eqref{eq:pbhdsigma}, we have,
  \begin{equation*}
    \norm{P_{\B_h}de_\sigma} \lesssim
    \chi_3\norm{e_\sigma}+\chi_4\norm{de_u}+\chi_5\norm{e_u}.
  \end{equation*}
  Combining all these estimates, we get the estimate in the claim.
\end{proof}

\subsection{Proof of Improved Error Estimates Theorem}
The estimates in Theorem~\ref{thm:improved_est} are derived from the four
preliminary estimates Lemma \ref{lem:pdsigma}, Lemma \ref{lem:pdu}, Lemma
\ref{lem:pu}, and Lemma \ref{lem:psigma}. Using $1$ for quantities which
are bounded and $\epsilon$ for quantities which goes to zero as
$h\rightarrow 0$, the four preliminary estimates has the following
structure:
\begin{equation*}
  \begin{bmatrix}
    \norm{de_\sigma}\\
    \norm{de_u}\\
    \norm{e_\sigma}\\
    \norm{e_u}
  \end{bmatrix}
  \lesssim
  \begin{bmatrix}
    0        &        1 &        1 &        1 \\
    \epsilon &        0 &        1 &        1 \\
    \epsilon & \epsilon &        0 &        1 \\
    \epsilon & \epsilon & \epsilon &        0 
  \end{bmatrix}
  \begin{bmatrix}
    \norm{de_\sigma}\\
    \norm{de_u}\\
    \norm{e_\sigma}\\
    \norm{e_u}
  \end{bmatrix}
  +
  \begin{bmatrix}
    \norm{dE_\sigma}\\
    \norm{dE_u}\\
    \norm{E_\sigma}\\
    \norm{E_u}
  \end{bmatrix}
\end{equation*}
For example, we can substitute the first line into the second and hide the
$\epsilon\norm{de_u}$ term in the left-hand side of the second line
assuming $h$ is sufficiently small. This, in effect, switches
$\norm{de_\sigma}$ in the second line to $\norm{dE_\sigma}$. Due the
vanishing diangonal and epsilon lower-triangle of the matrix above, this
procedure can be repeated to eliminate all unknown error terms like
$\norm{e_u}$ and switch them with known error terms like
$\norm{E_u}$. After this linear algebra exercise, we get the estimates in
Theorem~\ref{thm:improved_est}.

\section{Numerical examples}
\label{sec:numexp}
In this section, we show through numerical examples that the error rates
given by Theorem~\ref{thm:improved_est} in Table~\ref{tb:conv} are in fact
achieved and cannot be improved.

In 3D, there are four cases of the Hodge Laplace problems for differential
forms of degree $0,1,2,3$. The $0$-form and $3$-form cases lead to the
scalar Laplace problem in the non-mixed form and mixed form respectively.
The numerical results in these two cases are well-known \cite{Schatz1974,
  Douglas1982, Douglas1985} and will not be duplicated here. We focus on
the $1$-form and $2$-form case.

Let $\Omega=[0,1]^3$ be the unit cube in $\sn{R}^3$. The $1$-form mixed
Hodge Laplace problem with natural boundary conditions is: given
$f\in L^2$, find $u\in D$ satisfying:
\begin{gather*}
  (\grad+l_3)(-\Div+l_2) u + \curl(\curl+l_1)u + l_4\curl u + l_5u=f,
  \qquad  \text{in $\Omega$}, \\
  u\cdot n = 0, \qquad (\curl u + l_1 u) \times n = 0,
  \qquad  \text{on $\partial\Omega$}.
\end{gather*}
We choose the following smooth function as the exact solution:
\begin{equation*}
  u =
  \begin{bmatrix}
    \sin \pi x \cos \pi z \\
    \cos \pi x \sin \pi y \\
    \cos \pi y \sin \pi z
  \end{bmatrix},
\end{equation*}
with the coefficients:
\begin{gather*}
  l_1 =
  \begin{bmatrix}
    \sin\pi y \sin\pi z & \sin\pi z & \sin\pi y \\
    \sin\pi z & \sin\pi x \sin\pi z & \sin\pi x \\
    \sin\pi y & \sin\pi x & \sin\pi x \sin\pi y
  \end{bmatrix}, \quad
  l_2 =
  \begin{bmatrix}
    2 \\ 0 \\ 1
  \end{bmatrix}, \\
  l_3 =
  \begin{bmatrix}
    3 \\ 2 \\ -1
  \end{bmatrix}, \quad
  l_4 =
  \begin{bmatrix}
    1 & 2 &-1 \\ 
    3 &-3 &-3 \\
    1 & 3 & 1    
  \end{bmatrix}, \quad
  l_5 =
  \begin{bmatrix}
    10 & 0 & 0 \\
    0  &10 & 0 \\
    0  & 0 & 0 
  \end{bmatrix}.
\end{gather*}
At the discrete level, since $\sigma=(-\Div+l_2)u$ is a $0$-form where
$\mathcal{P}_{r}\Lambda^0=\mathcal{P}^-_{r}\Lambda^0$, we only have two
element pairs: $\mathcal{P}_{r}\Lambda^0\times \mathcal{P}_{r-1}\Lambda^1$
and $\mathcal{P}_{r}\Lambda^0\times \mathcal{P}^-_r\Lambda^1$.

On the same domain $\Omega=[0,1]^3$ as before, the $2$-form mixed Hodge
Laplace problem with natural boundary conditions is: given
$f\in L^2$, find $u\in D$ satisfying:
\begin{gather*}
  (\curl+l_3)(\curl+l_2) u -\grad (\Div+l_1)u + l_4\Div u + l_5u.  
  \qquad  \text{in $\Omega$}, \\
  u\times n = 0, \qquad (\Div + l_1) u = 0,
  \qquad  \text{on $\partial\Omega$}.
\end{gather*}
We choose the following smooth function as the exact solution in this case:
\begin{equation*}
  u =
  \begin{bmatrix}
    (\cos\pi x+3)\sin\pi y \sin\pi z \\
    \sin\pi x(\cos\pi y+2)\sin\pi z \\
    \sin\pi x \sin\pi y(\cos\pi z+2)
  \end{bmatrix},
\end{equation*}
with the coefficients
\begin{gather*}
  l_1 =
  \begin{bmatrix}
    \sin\pi x \\ -\sin\pi y \\ 0
  \end{bmatrix}, \qquad
  l_2 =
  \begin{bmatrix}
    1 & 2 &-1 \\
    2 &-2 & 0 \\
    1 & 3 & 1
  \end{bmatrix}, \\
  l_3 =
  \begin{bmatrix}
    1 & 0 &-1 \\
    0 &-1 & 0 \\
    1 & 2 & 1
  \end{bmatrix}, \quad
  l_4 =
  \begin{bmatrix}
    1 \\ 2 \\ -1
  \end{bmatrix}, \quad
  l_5 =
  \begin{bmatrix}
    10 & 0 & 0 \\
    0  &10 & 0 \\
    0  & 0 & 0 
  \end{bmatrix}.
\end{gather*}
At the discrete level, we have all four canonical pairs: 
\begin{equation*}
  \mathcal{P}_{r+1}\Lambda^1\times \mathcal{P}_{r}\Lambda^2, \quad
  \mathcal{P}^-_{r+1}\Lambda^1\times \mathcal{P}_{r}\Lambda^2,  \quad
  \mathcal{P}_{r}\Lambda^1\times \mathcal{P}_{r}^-\Lambda^2,  \quad
  \mathcal{P}^-_{r}\Lambda^1\times \mathcal{P}^-_{r}\Lambda^2.
\end{equation*}

In all the numerical experiments, we obtain a quasi-uniform triangulation
of size $m$ for the unit cube $\Omega$ by first triangulating it uniformly
with an $m\times m\times m$ mesh and then perturbing each interior mesh
node randomly within $20\%$ of the mesh size $1/m$ in all three coordinate
directions.

All four pairs of the canonical FEEC elements in dimension $\leq 3$ of all
degrees are supported by the open source finite element package FEniCS
\cite{Logg2012}, in which all our numerical codes are implemented.

For example, for the unperturbed $1$-form problem, with
$\mathcal{P}_{2}\Lambda^1\times \mathcal{P}_{1}\Lambda^2$, we get
\begin{center}
  \begin{tabular}{ r | r  r | r  r | r  r | r  r }
    \hline
    m 
    & $\norm{\sigma-\sigma_h}$   & rate
    & $\norm{d(\sigma-\sigma_h)}$ & rate
    & $\norm{u-u_h}$             & rate
    & $\norm{d(u-u_h)}$          & rate \\\hline
    2 & 2.766e-01 &      & 2.362e+00 &      & 2.244e-01 &      & 1.441e+00 &\\
    4 & 3.940e-02 & 2.37 & 8.529e-01 & 1.24 & 6.961e-02 & 1.43 & 7.434e-01 & 0.81\\
    8 & 4.504e-03 & 3.20 & 2.312e-01 & 1.93 & 1.859e-02 & 1.95 & 3.740e-01 & 1.01\\
    16& 5.208e-04 & 2.98 & 5.716e-02 & 1.93 & 4.659e-03 & 1.91 & 1.868e-01 & 0.96\\
    \hline
  \end{tabular}  
\end{center}
Thus, for example, $\norm{\sigma-\sigma_h}_{L^2}=2.766\times 10^{-1}$ on a
$2\times2\times 2$ mesh. The rates are computed between the two successive
errors. The optimal rates of $3,2,2,1$ for $\sigma,d\sigma,u,du$ are clear.
With an $l_4$ lower-order perturbation, we get:
\begin{center}
  \begin{tabular}{ r | r  r | r  r | r  r | r  r }
    \hline
    m 
    & $\norm{\sigma-\sigma_h}$   & rate
    & $\norm{d(\sigma-\sigma_h)}$ & rate
    & $\norm{u-u_h}$             & rate
    & $\norm{d(u-u_h)}$          & rate \\\hline
    2 & 4.190e-01 &      & 3.522e+00 &      & 2.215e-01 &      & 1.434e+00
&\\
    4 & 1.259e-01 & 1.44 & 1.583e+00 & 0.96 & 6.773e-02 & 1.42 & 7.335e-01 & 0.81\\
    8 & 3.689e-02 & 1.54 & 7.393e-01 & 0.96 & 1.874e-02 & 1.61 & 3.749e-01 & 0.84\\
    16& 9.968e-03 & 1.86 & 3.595e-01 & 1.02 & 4.799e-03 & 1.93 & 1.868e-01 & 0.99\\
    \hline
  \end{tabular}  
\end{center}
Clearly the convergence rates for $\sigma$ and $d\sigma$ in $L^2$ are
reduced by $1$ as predicted.

There are too many cases for us to list all the detailed results. We
instead only summarize the numerical results here. First, the error rates
in Table~\ref{tb:conv} are guaranteed for all FEEC elements for the Hodge
Laplace problem. Second, for each case with a reduction in the error rates
predicted in Table~\ref{tb:conv}, there is at least one Hodge Laplace
problem with a certain form degree which can only converge at that reduced
rate. In this sense, the rates in Table~\ref{tb:conv} is optimal. However,
we note that the rates in Table~\ref{tb:conv} do no represent an upper
bound for all possible cases. For example, when the $l_5$ term is given by
multiplication by a smooth scalar coefficient, we do not observe a
reduction of convergence rates in the $L^2$-error of $\sigma$. We also
observed that for $1$-forms in 3D, an $l_2$-term given by multiplication by
a generic smooth coefficient do not degrade the $L^2$-error rate in
$\sigma$.

The full numerical results along with the python source code used in FEniCS
can be found at the companion code repository at
\begin{center}
  \url{https://bitbucket.org/lzlarryli/feeclotexp}.
\end{center}
We note that due to the randomness involved (random perturbation applied to
the mesh), the error numbers will not be exactly the same but very close to
what we have listed here.

\bibliographystyle{amsplain}

% Below is generated by BibTeX

\providecommand{\bysame}{\leavevmode\hbox to3em{\hrulefill}\thinspace}
\providecommand{\MR}{\relax\ifhmode\unskip\space\fi MR }
% \MRhref is called by the amsart/book/proc definition of \MR.
\providecommand{\MRhref}[2]{%
  \href{http://www.ams.org/mathscinet-getitem?mr=#1}{#2}
}
\providecommand{\href}[2]{#2}

\end{document}